\documentclass[a4paper,oneside]{article}

\usepackage{latexsym,bm}
\usepackage{mathrsfs,amsmath,amssymb,amsthm}
\usepackage[dvipdfm]{graphicx}

\usepackage[english]{babel}
\usepackage[latin1]{inputenc}
\usepackage{tikz}
\usetikzlibrary{matrix, arrows, shapes} 
\usetikzlibrary{decorations.pathmorphing}

\usepackage{enumerate}
\usepackage{caption}
\usepackage{subcaption}
\usepackage{wrapfig}
\usepackage{float}
\usepackage{courier}
\theoremstyle{plain}

\newtheorem{theorem}{\textsc{Theorem}}
\newtheorem{lemma}{\textsc{Lemma}}
\newtheorem{proposition}{\textsc{Proposition}}
\newtheorem*{proposition*}{\textsc{Proposition}}
\newtheorem*{maintheorem}{\textsc{Main Theorem}}

\newcommand{\RR}{\mathbb{R}}		

\newcommand*\aabbcc{\includegraphics[scale=0.1]{aabbcc}}
\newcommand*\aabccb{\includegraphics[scale=0.1]{aabccb}}
\newcommand*\aabcbc{\includegraphics[scale=0.1]{aabcbc}}

\theoremstyle{definition}
\newtheorem*{definition}{Definition}

\title{On the Blow-analytic Equivalence of Tribranched Plane Curves}

\author{Cristina \textsc{Valle}\thanks{The author is supported by a Japanese Government scholarship (Monbukagakusho).}}

\date{}

\begin{document}

\label{startpage}

\maketitle

\abstract{We prove the finiteness of the number of blow-analytic equivalence classes of embedded plane curve germs for any fixed number of branches and for any fixed value of $\mu'$ ---a combinatorial invariant coming from the dual graphs of good resolutions of embedded plane curve singularities. In order to do so, we develop the concept of standard form of a dual graph. We show that, fixed $\mu'$ in $\mathbb{N}$, there are only a finite number of standard forms, and to each one of them correspond a finite number of blow-analytic equivalence classes. In the tribranched case, we are able to give an explicit upper bound to the number of graph standard forms. For $\mu'\leq 2$, we also provide a complete list of standard forms.}

\section{Introduction}
Compared to complex algebraic geometry, real algebraic geometry has a flexible nature, thus offering many different approaches to the subject. For example, in the real case there is no analogue of Hartogs' extension theorem. This means, in particular, that given an analytic function germ $f: (\RR^2, 0) \rightarrow \RR$, we can construct an analytically distinct $g: (\RR^2, 0) \rightarrow \RR$ such that $f$ and $g$ are isomorphic on a complement of the origin.

Thus, for a qualitative study of real singularities, we have a chance to introduce a suitable equivalence relation taking into account this phenomenon. Blow-analytic equivalence, originally introduced by Kuo \cite{original}, is one possibility. 

For function germs, several classification results are known, cf. \cite{Koike} \cite{Fukui}. In this article, we shall study the classification of germs of \textit{embedded} plane curve singularities up to blow-analytic homeomorphism, instead.

The first result in this direction appeared in \cite{Kuo} in 1998 and it showed how a straight line is blow-analytically equivalent to a cusp (in fact, it is equivalent to any unibranched germ of plane curve, no matter how singular). The peculiar behaviour of this example ---known as the Kobayashi--Kuo example--- suggests that the blow-analytic equivalence of embedded curves has strong topological nature. 

Here we aim to continue and expand the blow-analytic classification of embedded plane curve singularities. Firstly, we recall the basic definitions and constructions ---which can be found in \cite{Kuo} \cite{Kob}--- along with previous results in the subject, mostly concerning the classification of unibranched and bibranched germs of plane curves.

The study of the tribranched case occupies sections four and five, where we compute an upper estimate for the number of blow-analytic equivalence classes depending on a combinatorial invariant $\mu'$ defined in section two, and then we proceed to find representatives for each class for $\mu' \leq 2$.

Our main result concerns the general $n$-branched case and is stated in section three. We show the ``local'' finiteness of the blow-analytic classification, that is, the finiteness of the number of blow-analytic equivalence classes for fixed values of the number of branches and of the above mentioned discrete invariant.

We conclude this article with some topological considerations about the tools used, which will be further addressed in a future study \cite{new}.\\

We would like to thank the referee for the useful comments and suggestions made during the revision process.

\section{Blow-analytic invariants}
\begin{definition}
A function $f: \RR^2 \rightarrow \RR$ is \textit{blow-analytic} if there exists a composition of blow-ups $\beta = \beta_1 \circ \cdots \circ \beta_m : X \rightarrow \RR^2$ such that the composition $f \circ \beta$ is analytic.

A homeomorphism $h: \RR^2 \rightarrow \RR^2$ is \textit{blow-analytic} if $h$ and its inverse $h^{-1}$ both have blow-analytic components.
\end{definition}

In this article, we study the classification of embedded germs of real plane curves up to blow-analytic homeomorphism. We say that two germs of curves $(C, 0)$, $(D, 0)$ in $\RR^2$ with an isolated singularity at the origin are \textit{blow-analytically equivalent} if there exist a blow-analytic homeomorphism $h: \RR^2 \rightarrow \RR^2$ that carries $(C, 0)$ to $(D, 0)$.\\

It follows from the definition that two analytically equivalent curves are also blow-analytically equivalent. On the other hand, there are known examples of blow-analytically equivalent curves that are not even $\mathcal{C}^1$-equivalent or bi-Lipschitz equivalent (\cite{Koike}). In this sense, the classification of singular curves up to blow-analytic homeomorphism offers a more flexible approach than the analytic one.\\	

Now, let $X$ be a surface which is a tubular neighbourhood of the union of compact smooth curves $\{E_j\}_{j=1}^m$ intersecting transversally. We construct the weighted dual graph $\varGamma$ associated to $X$ by drawing a vertex $v_i$ for each central curve $E_i$, and connecting two vertices by an edge if and only if the corresponding curves intersect. To each vertex we assign as weight the $\mathbb{Z}/2\mathbb{Z}$-valued self-intersection number of the corresponding curve. In figures of the graphs $\varGamma$, we represent odd curves as white vertices and even curves as black vertices.

\begin{definition}
We say that $X$ is \textit{smoothly contractible} if it is a surface obtained from $(\RR^2, 0)$ by a finite sequence of blow-ups and blow-downs.
\end{definition}

It follows easily from the definition that if $X$ is smoothly contractible, then its dual graph $\varGamma$ is a tree.

Let $A = (a_{ij})$ be the $\mathbb{Z}/2\mathbb{Z}$-valued intersection matrix associated to $\varGamma$, i.e., the matrix whose entries are the $\mathbb{Z}/2\mathbb{Z}$-valued intersection numbers $a_{ij} = E_i \cdot E_j$. It has been proven in \cite{Kuo} that $X$ is smoothly contractible if and only if the determinant of $A$ is $1$. Since the information about the intersection matrix is captured by $\varGamma$, we also say that $\varGamma$ is \textit{smoothly contractible} if and only if the determinant of $A$ is $1$.\\

In the unibranched case, this invariant has been used to completely classify plane curve singularities up to blow-analytic homeomorphism.

\begin{theorem}[\cite{Kuo}]
All unibranched germs of plane curves are blow-analytically equivalent to a line.
\end{theorem}

Next, assume that $(C,0)$ has more than one branches and set $C = \bigcup_{i=1}^n C_i$ its irreducible decomposition. Let $X$ be a good resolution of $\RR^2$ at the origin, i.e., an embedded resolution which is a composition of successive blow-ups and blow-downs such that the support of the total transform of $C$ is simple normal crossing. We define $\varGamma^*$ to be the extension of $\varGamma$ obtained by adding a vertex for each component of the strict transform and an edge where a non-compact component (i.e., a component of the strict transform) intersects an exceptional curve. It follows from the goodness of the resolution that $\varGamma^*$ is a tree.

Blow-analytic equivalence of curve germs determines an equivalence relation for triplets $(X,\, \cup^n_{i=1} \tilde{C_i}, \, \cup_j E_j)$ (where $\tilde{C_i}$ is the strict transform of $C_i$), which induces an equivalence relation for trees $\varGamma$ and $\varGamma^*$. In case of ambiguity, we specify which curve germ corresponds to the dual graphs by writing $\varGamma(C)$ and $\varGamma^*(C)$.\\

Let $C = \bigcup_{i=1}^n C_i$ and $C' = \bigcup_{i=1}^n C'_i$ be two blow-analytically equivalent plane curve germs, then the blow-analytic homeomorphism $h:(\RR^2, \, C, \,0) \rightarrow (\RR^2,\, C', \,0)$ induces a bijection $\bar{h}: \{1, 2, \, \ldots \,, n\} \rightarrow \{1, 2, \,\ldots \,, n\}$ such that $h(C_i) = C'_{\bar{h}(i)}$.

Let $(X, \, \cup^n_{i=1} \tilde{C_i}, \, \cup_j E_j)$ and $(X', \, \cup^n_{i=1} \tilde{C'_i}, \, \cup_{j'} E'_{j'})$ be good embedded resolutions of $(C, 0)$ and $(C', 0)$ respectively, and let $(\tilde{X}, \, \cup^n_{i=1} \tilde{\tilde{C_i}}, \, \cup_{\tilde{j}} \tilde{E}_{\tilde{j}})$ be a common good resolution which dominates $(X, \, \cup^n_{i=1} \tilde{C_i}, \, \cup_j E_j)$ and $(X', \, \cup^n_{i=1} \tilde{C'_i}, \, \cup_{j'} E'_{j'})$. Consider a path $\gamma_{ij}$ in the exceptional set of $X$ connecting the strict transforms of $C_i$ and $C_j$, with $i\neq j$. We restrict ourselves to minimal paths, i.e., those $\gamma_{ij}$ which, amongst all paths connecting the strict transforms of $C_i$ and $C_j$, go through the minimum number of exceptional curves. Each path $\gamma_{ij}$ has a lift in $\tilde{X}$, and thus an image $\gamma'_{\bar{h}(i)\bar{h}(j)}$ in $X'$.

In the dual graph, $\gamma_{ij}$ (resp. $\gamma'_{\bar{h}(i)\bar{h}(j)}$) determines a path $\gamma^*_{ij}$ (resp. $(\gamma'_{\bar{h}(i)\bar{h}(j)})^*$) in $\varGamma^*(C)$ (resp. $\varGamma^*(C')$) between the vertices corresponding to the strict transforms of $C_i$ and $C_j$ (resp. $C'_{\bar{h}(i)}$ and $C'_{\bar{h}(j)}$). Since $\varGamma^*$ is a tree, for fixed $i,j$ there is a unique path $\gamma^*_{ij}$ in the dual graph corresponding to all minimal paths $\gamma_{ij}$ in the resolution.

\begin{lemma}\label{paths}
Let $\varGamma_{ij}(C)$ (resp. $\varGamma^*_{ij}(C)$) be the graph obtained by removing all vertices in $\gamma^*_{ij}$ and the connecting edges from $\varGamma(C)$ (resp. $\varGamma^*(C)$), and let $\Delta_{ij}(C)$ be the set of connected components $G$ in $\varGamma_{ij}(C)$ such that $\mu(G)\neq 0$, where $\mu(G)$ is the corank of the $\mathbb{Z}/2\mathbb{Z}$-valued intersection matrix associated to $G$. Let $\Delta^*_{ij}(C)$ denote the natural extension of $\Delta_{ij}(C)$ in $\varGamma^*(C)$.

Then, a blow-analytic homeomorphism $h:(\RR^2,\, C, \,0) \rightarrow (\RR^2,\, C',\, 0)$ induces a bijection $\hat{h}: \Delta^*_{ij}(C) \rightarrow \Delta^*_{\bar{h}(i)\bar{h}(j)}(C')$. In particular, $\mu(\varGamma_{ij}(C)) = \mu(\varGamma_{\bar{h}(i)\bar{h}(j)}(C'))$.
\end{lemma}

\begin{proof} As before, take a good resolution $(\tilde{X}, \, \cup^n_{i=1} \tilde{\tilde{C_i}}, \, \cup_{\tilde{j}} \tilde{E}_{\tilde{j}})$ dominating $(X, \, \cup^n_{i=1} \tilde{C_i}, \, \cup_j E_j)$ and $(X', \, \cup^n_{i=1} \tilde{C'_i}, \, \cup_{j'} E'_{j'})$, and let $\beta_k$ be a step in the sequence of blow-ups from $X$ to $\tilde{X}$.

For any $k$, if the centre of $\beta_k$ is a point in $\gamma_{ij}$, then the exceptional curve $E_k$ intersects the lift of $\gamma_{ij}$, therefore it does not contribute to $\varGamma^*_{ij}(C)$.

If the centre of $\beta_k$ is not in $\gamma_{ij}$ but on a curve belonging to $\gamma^*_{ij}$, after the blow-up an isolated odd vertex is added to $\varGamma^*_{ij}(C)$, thus creating a new connected component which is smoothly contractible and does not contribute to $\Delta_{ij}(C)$.

Finally, if the centre of $\beta_k$ is not on any curve in $\gamma^*_{ij}$, the vertex corresponding to the exceptional curve $E_k$ extends one of the connected components $G$ of $\varGamma_{ij}(C)$. Let us call $G'$ the extended component. Slightly abusing the notation, let $A' =~ (E'_p \cdot E'_q )$ be the $\mathbb{Z}/2\mathbb{Z}$-valued intersection matrix associated to $G'$. By a change of basis, $A' \approx_{\mathbb{Z}}~ \left( \begin{array}{cc}
1 & 0 \\
0 & A \end{array} \right)$, where $A$ is the $\mathbb{Z}/2\mathbb{Z}$-valued intersection matrix associated to $G$. Clearly, $\mu(G') = \mu(G)$, so $\beta_k$ preserves the corank of the connected components of $\varGamma_{ij}(C)$.

Since this holds for each step $\beta_k$ in the blow-up sequence, there exists a bijection between the elements of $\Delta^*_{ij}(C)$ and the elements of its lift in $\tilde{X}$. Furthermore, since $\gamma'_{\bar{h}(i)\bar{h}(j)}$ is the image of $\gamma_{ij}$ in $X'$, we have a bijection $\hat{h}: \Delta^*_{ij}(C) \rightarrow \Delta^*_{\bar{h}(i)\bar{h}(j)}(C')$. \phantom{blah}
\end{proof}

Let $I = \{ I_k : k=1, \, \ldots \,,p\}$ denote a partition of $\{1, 2, \, \ldots \,, n\}$ (i.e., $I_1 \cup \cdots \cup I_p = \{1, \,\ldots \,, n\}$ and $I_k\cap I_l=\emptyset$ if $1\leq k < l \leq p$). By considering the union of minimal paths $\gamma_{ij}$ between two components $\tilde{C}_i$ and $\tilde{C}_j$ of the strict transform with $i, j\in I_k$ and $i \neq j$, the proof of the above lemma can be generalised to a partition on the set of strict transform components. Namely, we have that a blow-analytic homeomorphism $h:(\RR^2,\, C, \,0) \rightarrow (\RR^2,\, C', \,0)$ induces a bijection $\hat{h}: \Delta_I(C) \rightarrow \Delta_{\bar{h}(I)}(C')$ (where $\bar{h}(I) = \{\bar{h}(I_k) : k=1, \, \ldots \,, p\}$) and, in particular, the corank $\mu(\varGamma_I)$ is a blow-analytic invariant.\\

When $I = \{ \{1\}, \{2\}, \, \ldots \,, \{n\} \}$, $\mu_I$ generalises the invariant $\mu$ defined in \cite{Kob}, although the author uses a different method to prove its invariance.

When $I = \{1, 2,\, \ldots \,, n\}$, we write $\mu_I$ as $\mu'$ for convenience. As we shall see, the value of $\mu'$ bounds from below the least number of components in the exceptional divisor of a good resolution of any curve germ in that equivalence class.\\

In the case of bibranched singularities, $\mu'$ provides the following classification:

\begin{theorem}[\cite{Kob}]
Bibranched germs of plane curves have isomorphic resolution graphs if and only if they have the same $\mu'$.
\end{theorem}

Our aim is to study the classification of singularities with three branches. In the following sections, firstly we present some general results for $n$-branched germs and then focus on the tribranched case.

\section{Finiteness}

We approach the problem of the classification of embedded plane curve singularities by providing a classification of the dual graphs of their resolutions. Namely, given a smoothly contractible graph $\varGamma$, we perform blow-ups and blow-downs to simplify $\varGamma$ and reduce it to a standard form, without changing the blow-analytic equivalence class of the corresponding embedded curve germ $(C,0)$. This method allows us to make easy combinatorial computations and graphic representations.

Two blow-analytically equivalent germs have by definition a pair of isomorphic dual graphs. It should be noted that the converse is not true: in fact, we can explicitly construct examples of non-equivalent singularities with isomorphic dual graphs, as shown in the last section of this paper. However, to any dual graph correspond only a finite number of blow-analytically distinct embedded plane curve germs.\\

In what follows, we denote $Q$ an even vertex with valency $1$ in $\varGamma^*$, where the \textit{valency} of a vertex is the number of edges incident to it. We call a vertex \textit{extremal} if it has valency $1$ in $\varGamma$, and we call \textit{special vertex} a vertex with valency $3$ or more in $\varGamma^*$. We remark that a configuration is not smoothly contractible if $\varGamma$ contains two vertices of type $Q$ attached to the same vertex. In fact, if the graph contains such a part, then the determinant of the $\mathbb{Z}/{2\mathbb{Z}}$-valued intersection matrix associated to $\varGamma$ vanishes (\cite{Kuo}).\\

Let $X$ be a good resolution and $\varGamma^*$ its extended dual graph. The operations listed below are a composition of blow-ups and blow-downs of $X$, expressed for simplicity in the graph language.
\begin{itemize}
	\item[C1] (\textit{Contraction $1$}): contract an odd vertex with valency $1$ in $\varGamma^*$;
	\item[C2] (\textit{Contraction $2$}): contract an odd vertex with valency $2$ in $\varGamma^*$;
	\item[C3] (\textit{Contraction $3$}): remove two adjacent even vertices, each having valency at most $2$ in $\varGamma^*$, by first blowing up at the intersection of the two exceptional curves and then performing C2 three times (contracting the newly created exceptional curve last);
	\item[M1] (\textit{Modification $1$}): if a vertex of type $Q$ is attached to an odd vertex, change the parity of the latter as shown in \cite{Kob}; namely, perform a blow-up at the point where the even curve in $Q$ intersects the odd curve and then contract the extremal odd vertex.
\end{itemize}

Given a graph $\varGamma^*$ as above, perform contractions C1, C2 and C3 repeatedly, until no more contractions can be made. Since the size of the graph is finite and each contraction decreases the number of vertices in $\varGamma$, after a finite number of steps $\varGamma$ will be minimal under C1, C2 and C3. Next, apply M1 wherever it is possible. If $n=0$ (i.e., the embedded curve germ is an isolated point), the minimal graph under the above operations is an odd vertex with valency $0$, which we further contract, obtaining the empty graph.

The resulting surface $X$ is blow-analytically equivalent to the original one, and its dual graph is reduced to a \textit{standard form} of $\varGamma$.

\begin{proposition}
A standard form of $\varGamma$ satisfies the following properties:
\begin{itemize}
	\item[\normalfont{P1}] All non-special vertices are even;
	\item[\normalfont{P2}] All special vertices adjacent a vertex of type $Q$ are even;
	\item[\normalfont{P3}] The segment between two special vertices is at most one even vertex;
	\item[\normalfont{P4}] There are exactly $\mu'$ vertices of type $Q$.
\end{itemize}
\end{proposition}
\begin{proof}
P1 follows from the fact that any odd non-special vertex has been contracted by C1 or C2. P2 is a consequence of M1. P3 follows from P1 and by C3. To prove P4, set $I =\{1, 2, \, \ldots \, , n\}$ and consider the extremal vertices of each connected component in $\varGamma_I$. By P1, they are all even. Moreover, by C3 they can only be part of a path of length $1$ and two of them cannot be connected to the same vertex, since $\varGamma$ is smoothly contractible. Thus $\varGamma_I$ contains only even vertices, all disconnected, and exactly $\mu'$ of them, since each contributes to the corank by $1$.
\end{proof}

An arbitrary $\varGamma$ can always be reduced as shown above. Therefore, we shall restrict our attention to the easier task of classifying standard configurations.

\begin{proposition}\label{branches} 
For an $n$-branched embedded plane curve germ, a standard form of $\varGamma$ has at most $\mu' + n$ extremal vertices.
\end{proposition}
\begin{proof}
Since we assume $\varGamma$ to be a standard form, its extremal vertices must be either even vertices (corresponding to vertices of type $Q$) or vertices adjacent to at least a non-compact component. There are exactly $\mu'$ vertices of the first kind and at most $n$ of the second kind, thus there are at most $\mu' + n$ extremal vertices.
\end{proof}

\textit{Remark}. The number of extremities could be strictly less than $\mu' + n$. In fact, more than one non-compact components could be adjacent to the same extremal vertex, or it could also happen that some non-compact components are attached to non-extremal vertices.\\

We can gain additional information about standard forms by looking at the weights of extremal vertices and at the vertices to which they are connected.\\

Consider the case where an extremal vertex is adjacent to a non-compact component. If the extremal vertex is odd, then it must be adjacent to at least two non-compact components (otherwise it has valency $2$ in $\varGamma^*$ and can be smoothly contracted). 

Now, assume that the extremal vertex is even and adjacent to exactly one non-compact component. Then the preceding vertex $v$ must be a special vertex (if it were a even vertex with valency $2$, $\varGamma$ could be further reduced by C3 without losing normal crossingness) and, to avoid configurations which are not smoothly contractible, there cannot be a vertex of type $Q$ attached to $v$. Thus $v$ must be either adjacent to a non-compact component or have valency at least $3$ in $\varGamma$.

The above considerations prove the following lemma.
\begin{lemma}\label{extremal}
There are only four kinds of extremal vertices in a standard form:
\begin{itemize}
	\item vertices of type $Q$;
	\item vertices adjacent to at least two non-compact components;
	\item even vertices adjacent to exactly one non-compact component and preceded by another vertex adjacent to a non-compact component;
	\item even vertices adjacent to exactly one non-compact component and preceded by a vertex with valency at least $3$ in $\varGamma$.
\end{itemize}
\end{lemma}

We now prove our main result.

\begin{maintheorem}
The number of blow-analytic equivalence classes of $n$-branched germs of plane curves with $\mu' = k$ is finite for any $k$ in $\mathbb{N}$.
\end{maintheorem}
\begin{proof}
Given a germ of plane curve $(C,0)$, take a good resolution of the embedded singularity and consider its dual graph $\varGamma$. By the process described above, the tree $\varGamma$ can be reduced to its standard form, which, by Proposition \ref{branches}, has at most $k + n$ branches. Furthermore, the length of each branch is limited by the properties of standard forms. Since the number of trees of a finite size is finite, it follows that there are only a finite number of standard forms, given $n$ and $k$.

Observe that the number of smooth surfaces $X$ corresponding to a given standard form, as well as the number of choices for the positions of the $n$ non-compact components on $X$, is finite up to diffeomorphism. Thus, only finitely many blow-analytic equivalence classes of embedded plane curve singularities exist for fixed $n$ and $k$ in $\mathbb{N}$.
\end{proof}

\textit{Remark}. The theorem states the ``local'' finiteness of the blow-analytic classification, that is, for fixed values of $n$ and $\mu'$. Globally, the classification is infinite. In fact, it is easy to find standard configurations (and thus blow-analytic equivalence classes) for any number of branches and any $\mu'$ in $\mathbb{N}$.

\section{An upper bound}

While it is difficult to recover a generating formula for the exact number of blow-analytic equivalence classes given the number of branches $n$ and the value of $\mu'$, an upper bound to the number of standard forms can be estimated using combinatorial methods and some observations about the shape of $\varGamma$.

\begin{proposition}
In the tribranched case, the number of standard forms of $\varGamma$ with $\mu'=k$ is less than or equal to
\[
(k^3 -2k^2 -k +11) 2^{k-2}.
\]
\end{proposition}

\begin{proof}
Consider the topological structure of the minimal subtree connecting all non-compact components in $\varGamma$. We call this the \textit{trunk} of $\varGamma$.

In the tribranched case, there are only four possible shapes for the trunk of a standard form:

\begin{figure}[H]
\begin{subfigure}[h]{0.49\textwidth}
\centering
\includegraphics[scale=0.25]{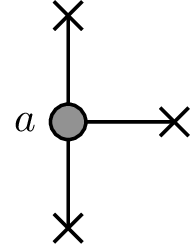}
\caption*{Type A}
\end{subfigure}
\begin{subfigure}[h]{0.49\textwidth}
\centering
\includegraphics[scale=0.25]{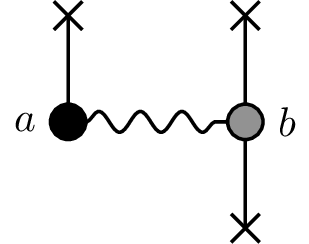}
\caption*{Type B}
\end{subfigure}

\begin{subfigure}[h]{0.49\textwidth}
\centering
\includegraphics[scale=0.25]{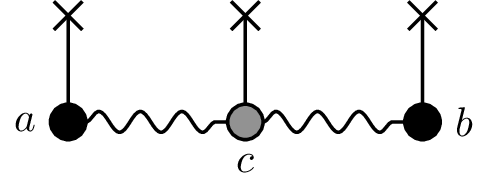}
\caption*{Type C}
\end{subfigure}
\begin{subfigure}[h]{0.49\textwidth}
\centering
\includegraphics[scale=0.25]{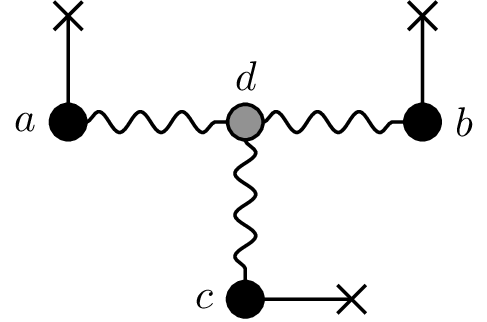}
\caption*{Type D}
\end{subfigure}
\end{figure}

\noindent where $\times$ represents a non-compact component, the grey vertices can be either even or odd exceptional curves, and waved edges between two vertices represent finite chains of even curves connecting them.\\

For $\mu' = k$, standard forms of $\varGamma$ can be obtained by adding $k$ vertices of type $Q$ to the trunks above.

Observe that, in order to avoid not smoothly contractible configurations, two vertices of type $Q$ cannot be attached to the same vertex. This implies that graphs of type A exist only for $k = 0$ and $k=1$.\\

In what follows assume $k>2$ for simplicity. The formula still holds for $k=0,1,2$, as shown by the computations in the next section.\\

\textit{Type B}. By Lemma \ref{extremal}, the vertex $a$ cannot be extremal, so there must be a vertex of type $Q$ attached to it.

If another vertex of type $Q$ is attached to $b$, the remaining $k-2$ vertices of type $Q$ must be placed in the middle. The segment between each pair of special vertices (if it exists) is at most one even vertex, so $2^{k-1}$ different configurations are obtained this way.

Similarly, if $b$ is extremal, then $k-1$ vertices of type $Q$ are attached to the edge of the trunk, which gives $2^k$ configurations.

Furthermore, $b$ can be either odd or even, so there are
\[
2 (2^{k-1} + 2^k) = 3 \cdot 2^k 
\]
configurations of type B.\\

\textit{Type C}. If vertices of type $Q$ are attached to both $a$ and $b$, there are
\[
\sum_{\alpha + \beta =k-2}(2^{\alpha+1}2^{\beta+1}) + \sum_{\alpha + \beta = k-3}(2^{\alpha+1}2^{\beta+1}) =
\]
\[
=(k-2)(k-1) 2^{k-1} + (k-2)(k-3) 2^{k-2} = (3k^2 -11k +10) 2^{k-2}
\]
configurations, where the two terms in the sum count separately whether there is a vertex of type $Q$ attached to $c$ or not.

On the other hand, if $b$ is an extremal vertex, then, by Lemma \ref{extremal}, the right edge is empty. Since $k>2$, there must be one vertex of type $Q$ attached to $a$ and the other $k-1$ to the left edge.
This gives
\[
2^{k-1} + 2^{k-2} = 3 \cdot 2^{k-2}
\]
new configurations.

Since $c$ can be either odd or even, the total number of configurations of type C is
\[
2 [(3k^2 -11k +10)2^{k-2} + 3 \cdot 2^{k-2}] = (3k^2 -11k +13) 2^{k-1}.
\]

\textit{Type D}. First consider the case in which $a$, $b$ and $c$ each have a vertex of type $Q$ attached to them. This gives
\[
\sum_{\alpha + \beta + \gamma = k-3}(2^{\alpha+1}2^{\beta+1}2^{\gamma+1}) + \sum_{\alpha + \beta + \gamma = k-4}(2^{\alpha+1}2^{\beta+1}2^{\gamma+1}) =
\]
\[
= \frac{1}{3}(k-3)(k-2)(k-1) 2^{k-1} + \frac{1}{3}(k-4)(k-3)(k-2) 2^{k-2}
\]
configurations.

Next, assume that $c$ is extremal. By Lemma \ref{extremal}, this means that the downward edge is empty and there cannot be vertices of type $Q$ attached to $d$, else not smoothly contractible configurations arise. So we have
\[
\sum_{\alpha + \beta = k-2}(2^{\alpha+1}2^{\beta+1}) = (k-2)(k-1)2^{k-1}
\]
configurations.

Observe that if two of the vertices in the trunk are extremal, Lemma \ref{extremal} implies that the corresponding edges are empty, which leads to configurations that are not smoothly contractible. So the previous two cases cover all possible configurations.

Since there are two colour choices for the vertex $d$, in total there are
\[
(k-2)(k^2 -3k +4)2^{k-1}
\]
configurations of type D.\\

Adding the numbers obtained for each type, we get the upper bound 
\[
(k^3 -2k^2 -k +11) 2^{k-2}.
\]
\end{proof}

\textit{Remark.} The above formula is merely an upper estimate of the number of standard forms for $\mu'=k$. In fact, the number includes some not smoothly contractible configurations as well as pairs of configurations which are blow-analytically equivalent (in the pair, one configuration is a standard form, to which the other can be reduced).

\section{Explicit classification of tribranched germs}

In this section, we restrict our attention to tribranched germs of plane curves and determine explicitly a standard form for each blow-analytic equivalence class, for low values of the invariant $\mu'$.

\begin{proposition}[\cite{Kob}] A germ of a tribranched plane curve with $\mu'=~0$ is blow-analytically equivalent to one of the following:
\begin{figure}[H]
\begin{subfigure}{0.49\textwidth}
\centering
\includegraphics[scale=0.25]{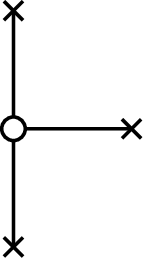}
\caption*{$(\{xy(x-y)=0\},0)$}
\end{subfigure}
\begin{subfigure}{0.49\textwidth}
\centering
\includegraphics[scale=0.25]{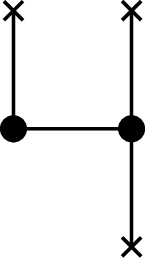}
\caption*{$(\{xy(x-y^2)=0\},0)$}
\end{subfigure}
\end{figure}
\end{proposition}

The following results provide a classification of the dual graphs of good resolutions with $\mu'=1, 2$.\\

Up to this point, we are not able to prove in general the uniqueness of standard forms in a given blow-analytic equivalence class. In Propositions \ref{k1} and \ref{k2}, we use the invariants $\mu_I$ to show that the standard forms listed in the statements are in fact blow-analytically distinct.\\

It may happen that two blow-analytically non-equivalent germs share the same graph standard form. However, to each standard form correspond at most a finite number of blow-analytic equivalence classes of plane curve germs, so we feel that a classification of the dual graphs is still a strong one from the blow-analytic point of view.

\begin{proposition}\label{k1}
The dual graph of any good resolution of a tribranched plane curve germ with $\mu' = 1$ is blow-analytically equivalent to exactly one of the following standard forms:
\begin{figure}[H]
\centering
\begin{subfigure}{0.49\textwidth}
\centering
\includegraphics[scale=0.25]{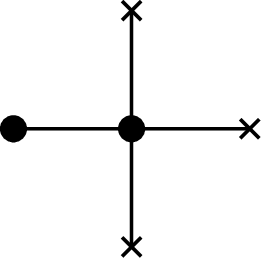}
\caption*{$A_2: (\{y(y-x^2)(y+x^2)=0\},0)$}
\end{subfigure}
\begin{subfigure}{0.49\textwidth}
\centering
\includegraphics[scale=0.25]{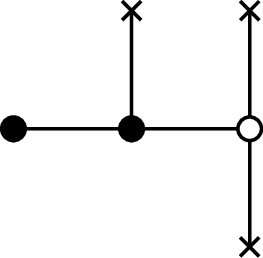}
\caption*{$B_1: (\{x(y - x)(y^2 - x^3) = 0\}, 0)$}
\end{subfigure}
\end{figure}

\begin{figure}[H]
\begin{subfigure}{0.49\textwidth}
\centering
\includegraphics[scale=0.25]{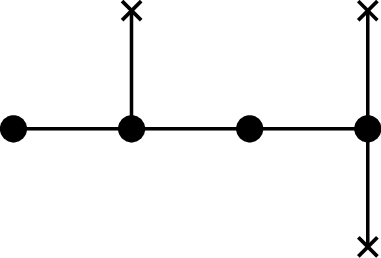}
\caption*{$B_4: (\{y(y - x^2)(y - x^4) = 0\}, 0)$}
\end{subfigure}
\begin{subfigure}{0.49\textwidth}
\centering
\includegraphics[scale=0.25]{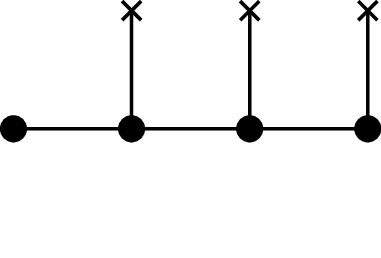}
\caption*{$C_2: (\{y(y - x^2)(y^2 - x^5) = 0\}, 0)$}
\end{subfigure}
\end{figure}
\end{proposition}

\begin{proof}
Consider a tribranched germ of plane curve $(C,0)$ and assume $\mu'=~1$. Take a good resolution of $(C,0)$, construct its dual graph $\varGamma$ and reduce it to a standard from as described in section two.

Since $(C,0)$ is tribranched, the trunk of the reduced $\varGamma$ must be of type $A$, $B$, $C$ or $D$. Furthermore, the assumption $\mu'=1$ implies that $\varGamma$ contains exactly one vertex of type $Q$.

Draw all configurations with $\mu'=1$ for each type, remembering that a segment between two special vertices is at most one even vertex and using Lemma \ref{extremal} for the extremal vertices. Then, $\varGamma$ must be blow-analytically equivalent to one of the following configurations:

\begin{figure}[H]
\begin{subfigure}{0.49\textwidth}
\centering
\includegraphics[scale=0.18]{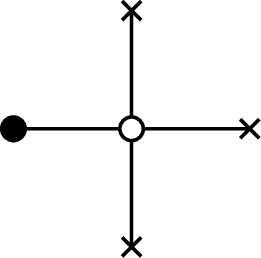}
\caption*{$A_1$}
\end{subfigure}
\begin{subfigure}{0.49\textwidth}
\centering
\includegraphics[scale=0.18]{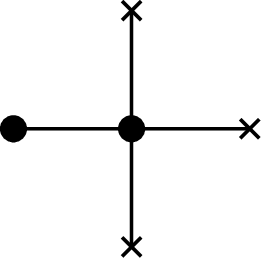}
\caption*{$A_2$}
\end{subfigure}

\begin{subfigure}{0.32\textwidth}
\centering
\includegraphics[scale=0.18]{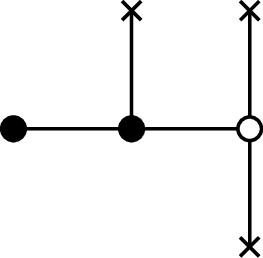}
\caption*{$B_1$}
\end{subfigure}
\begin{subfigure}{0.32\textwidth}
\centering
\includegraphics[scale=0.18]{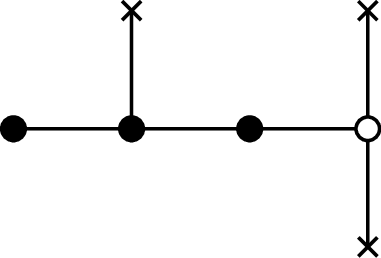}
\caption*{$B_3$}
\end{subfigure}
\begin{subfigure}{0.32\textwidth}
\centering
\includegraphics[scale=0.18]{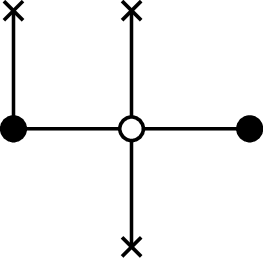}
\caption*{$B_5\; (\mu\neq 0)$}
\end{subfigure}

\begin{subfigure}{0.32\textwidth}
\centering
\includegraphics[scale=0.18]{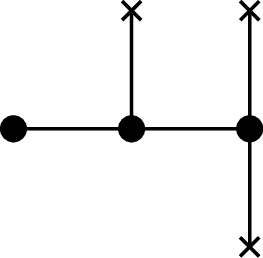}
\caption*{$B_2\; (\mu\neq 0)$}
\end{subfigure}
\begin{subfigure}{0.32\textwidth}
\centering
\includegraphics[scale=0.18]{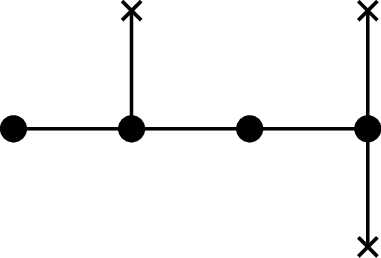}
\caption*{$B_4$}
\end{subfigure}
\begin{subfigure}{0.32\textwidth}
\centering
\includegraphics[scale=0.18]{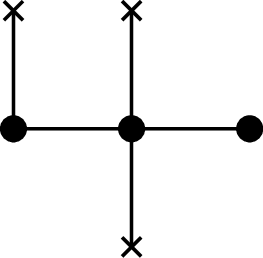}
\caption*{$B_6\; (\mu\neq 0)$}
\end{subfigure}

\begin{subfigure}{0.49\textwidth}
\centering
\includegraphics[scale=0.18]{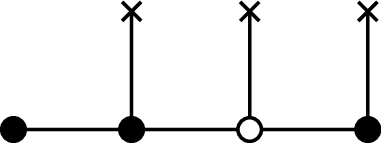}
\caption*{$C_1$}
\end{subfigure}
\begin{subfigure}{0.49\textwidth}
\centering
\includegraphics[scale=0.18]{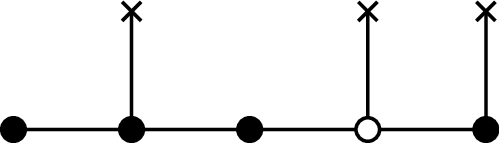}
\caption*{$C_3\; (\mu\neq 0)$}
\end{subfigure}

\begin{subfigure}{0.49\textwidth}
\centering
\includegraphics[scale=0.18]{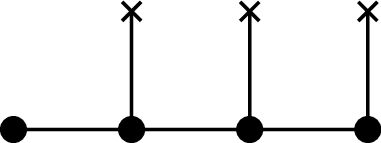}
\caption*{$C_2$}
\end{subfigure}
\begin{subfigure}{0.49\textwidth}
\centering
\includegraphics[scale=0.18]{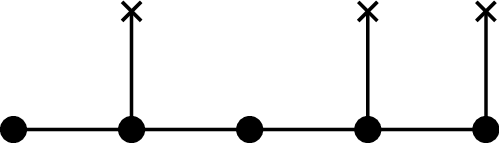}
\caption*{$C_4\; (\mu\neq 0)$}
\end{subfigure}
\end{figure}

Observe that all configurations of type $D$ with $\mu'=1$ are not smoothly contractible, thus cannot be the dual graph of a resolution. For the same reason, we also cross out of the list all configurations with $\mu\neq 0$.

For the remaining configurations, $A_1$, $B_3$ and $C_1$ are blow-analytically equivalent to $A_2$, $B_4$ and $C_2$ respectively. Only $4$ graphs are left and they are those of the statement.

Finally, the equation of a representative for each configuration can be found by contracting all exceptional curves (possibly performing blow-ups if no odd curves are present).

To show that the four configurations are blow-analytically distinct, label $\{1, 2, 3\}$ the vertices corresponding to the three non-compact components and consider the triplets $\{\mu(\varGamma_{12}), \mu(\varGamma_{13}), \mu(\varGamma_{23}) \}$, which are blow-analytic invariants by Lemma \ref{paths}. We have:
\begin{gather*}
A_2, B_4: \{1,1,1\} \qquad B_1:\{0,1,1\} \qquad C_2: \{0,1,2\}.
\end{gather*}
Since the coranks are not sufficient to distinguish between $A_2$ and $B_4$, we look explicitly at the sets $\Delta^*_{ij}$ for $1 \leq i < j \leq 3$:
\begin{align*}
\Delta^*_{12}(A_2) = \Delta^*_{13}(A_2) & = \Delta^*_{23}(A_2)= \lbrace \bullet, \times \rbrace ; \\
\Delta^*_{12}(B_4) = \Delta^*_{13}(B_4) & = \lbrace \bullet, \times \rbrace , \qquad \Delta^*_{23}(B_4) = \lbrace \includegraphics[scale=0.14]{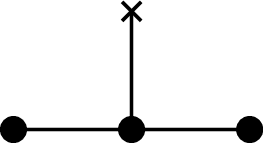} \rbrace.
\end{align*}

There is no bijection between $\Delta^*_{ij}(A_2)$ and $\Delta^*_{23}(B_4)$ for any choice of $ij$, so we conclude that no blow-analytic homeomorphism exists between plane curve germs having good resolutions equivalent to $A_2$ and $B_4$ respectively.
\end{proof}

\begin{proposition}\label{k2}
The dual graph of any good resolution of a tribranched plane curve germ with $\mu' = 2$ is blow-analytically equivalent to exactly one of the following standard forms:
\end{proposition}
\begin{figure}[H]
\begin{subfigure}{0.32\textwidth}
\centering
\includegraphics[scale=0.24]{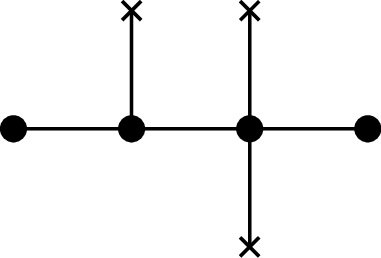}
\caption*{$B_2$}
\end{subfigure}
\begin{subfigure}{0.32\textwidth}
\centering
\includegraphics[scale=0.24]{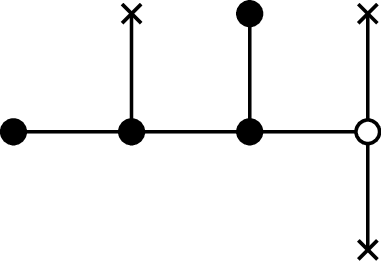}
\caption*{$B_5$}
\end{subfigure}
\begin{subfigure}{0.32\textwidth}
\centering
\includegraphics[scale=0.24]{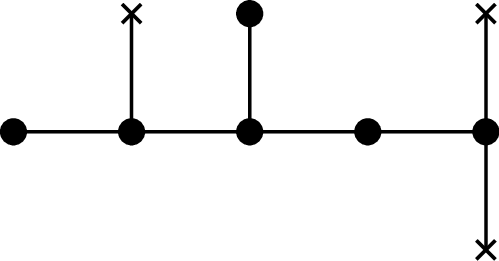}
\caption*{$B_{12}$}
\end{subfigure}
\end{figure}

\begin{figure}[H]
\begin{subfigure}{0.49\textwidth}
\centering
\includegraphics[scale=0.24]{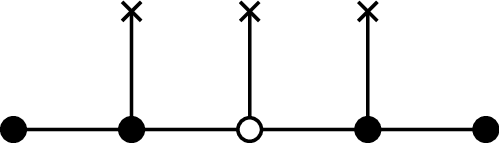}
\caption*{$C_1$}
\end{subfigure}
\begin{subfigure}{0.49\textwidth}
\centering
\includegraphics[scale=0.24]{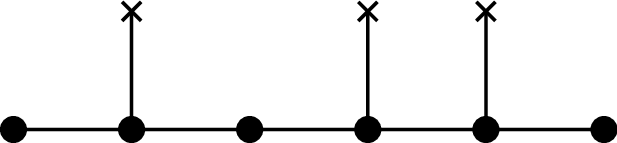}
\caption*{$C_4$}
\end{subfigure}

\begin{subfigure}{0.49\textwidth}
\centering
\includegraphics[scale=0.24]{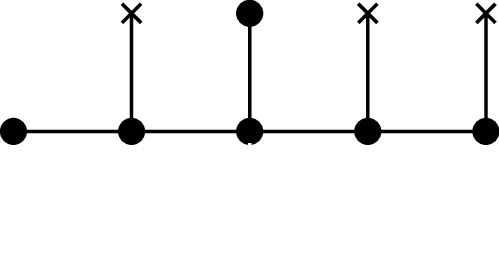}
\caption*{$C_8$}
\end{subfigure}
\begin{subfigure}{0.49\textwidth}
\centering
\includegraphics[scale=0.24]{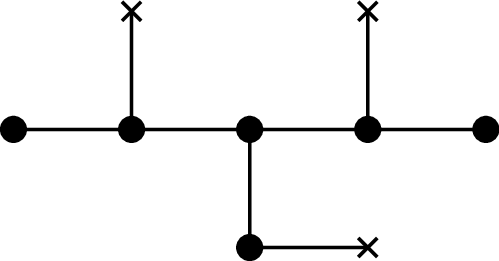}
\caption*{$D_2$}
\end{subfigure}
\end{figure}

\begin{proof}
The proof is similar to that of the previous proposition. For each type, draw all reduced configurations in which exactly two vertices of type $Q$ appear. The dual graph of any resolution of a tribranched singularity with $\mu'=2$ is blow-analytically equivalent to one of the graphs in the list below. Again, notice that there are no smoothly contractible configurations of type $A$.

\begin{figure}[H] 
\centering
\begin{subfigure}{0.32\textwidth}
\centering
\includegraphics[scale=0.18]{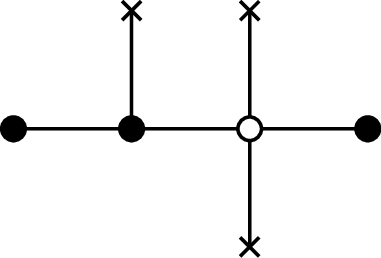}
\caption*{$B_1$}
\end{subfigure}
\begin{subfigure}{0.32\textwidth}
\centering
\includegraphics[scale=0.18]{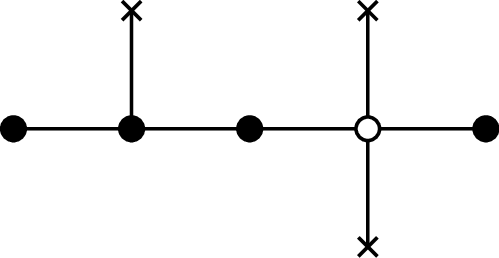}
\caption*{$B_3 \; (\mu\neq 0)$}
\end{subfigure}
\begin{subfigure}{0.32\textwidth}
\centering
\includegraphics[scale=0.18]{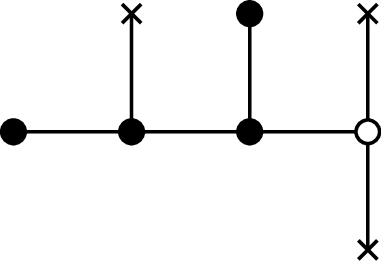}
\caption*{$B_5$}
\end{subfigure}

\begin{subfigure}{0.32\textwidth}
\centering
\includegraphics[scale=0.18]{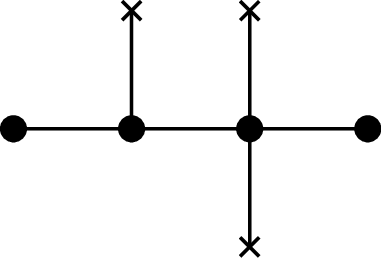}
\caption*{$B_2$}
\end{subfigure}
\begin{subfigure}{0.32\textwidth}
\centering
\includegraphics[scale=0.18]{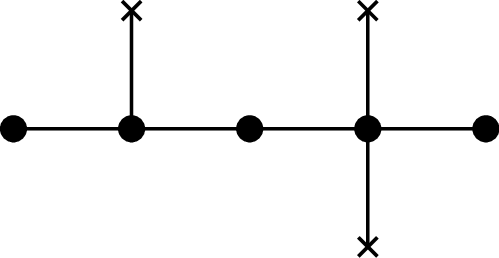}
\caption*{$B_4 \; (\mu\neq 0)$}
\end{subfigure}
\begin{subfigure}{0.32\textwidth}
\centering
\includegraphics[scale=0.18]{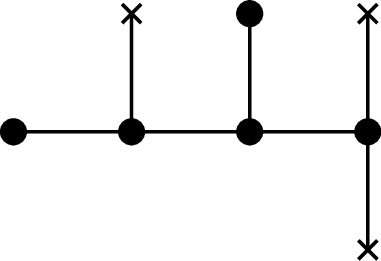}
\caption*{$B_6 \; (\mu\neq 0)$}
\end{subfigure}
\end{figure}

\begin{figure}[H] 
\centering
\begin{subfigure}{0.325\textwidth}
\centering
\includegraphics[scale=0.18]{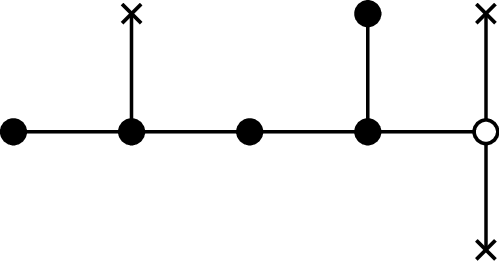}
\caption*{$B_7 \; (\mu\neq 0)$}
\end{subfigure}
\begin{subfigure}{0.325\textwidth}
\centering
\includegraphics[scale=0.18]{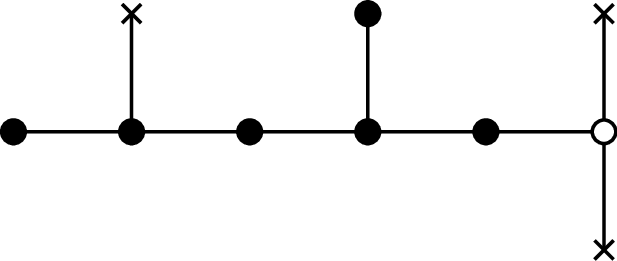}
\caption*{$B_9 \; (\mu\neq 0)$}
\end{subfigure}
\begin{subfigure}{0.325\textwidth}
\centering
\includegraphics[scale=0.18]{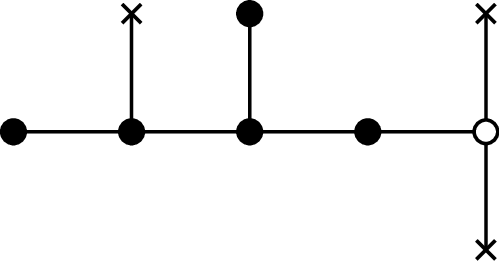}
\caption*{$B_{11}$}
\end{subfigure}

\begin{subfigure}{0.325\textwidth}
\centering
\includegraphics[scale=0.18]{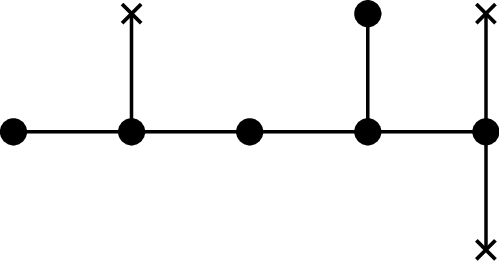}
\caption*{$B_8 \; (\mu\neq 0)$}
\end{subfigure}
\begin{subfigure}{0.325\textwidth}
\centering
\includegraphics[scale=0.18]{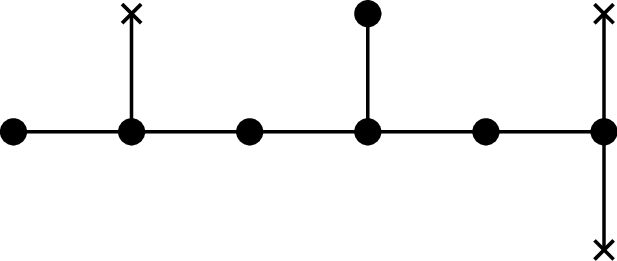}
\caption*{$B_{10} \; (\mu\neq 0)$}
\end{subfigure}
\begin{subfigure}{0.325\textwidth}
\centering
\includegraphics[scale=0.18]{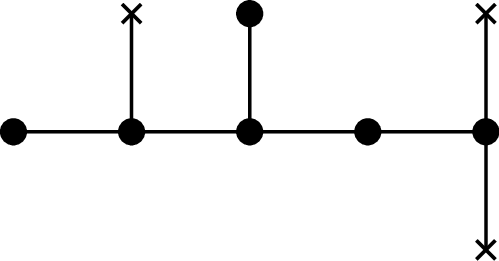}
\caption*{$B_{12}$}
\end{subfigure}
\end{figure}

\begin{figure}[H] 
\begin{subfigure}{0.25\textwidth}
\centering
\includegraphics[scale=0.18]{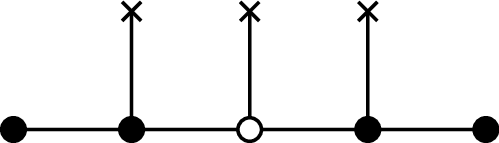}
\caption*{$C_1$}
\end{subfigure}
\begin{subfigure}{0.35\textwidth}
\centering
\includegraphics[scale=0.18]{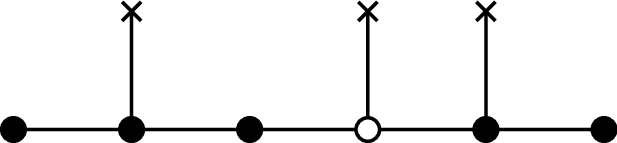}
\caption*{$C_3$}
\end{subfigure}
\begin{subfigure}{0.35\textwidth}
\centering
\includegraphics[scale=0.18]{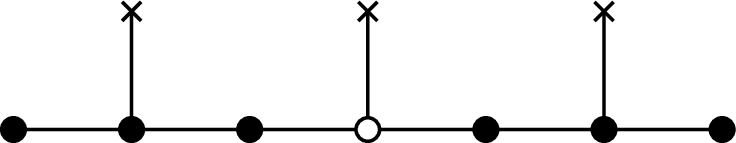}
\caption*{$C_5 \; (\mu\neq 0)$}
\end{subfigure}

\begin{subfigure}{0.25\textwidth}
\centering
\includegraphics[scale=0.18]{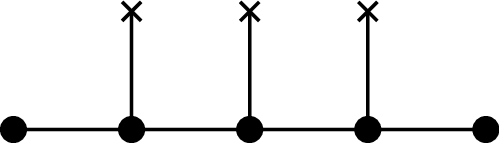}
\caption*{$C_2 \; (\mu\neq 0)$}
\end{subfigure}
\begin{subfigure}{0.35\textwidth}
\centering
\includegraphics[scale=0.18]{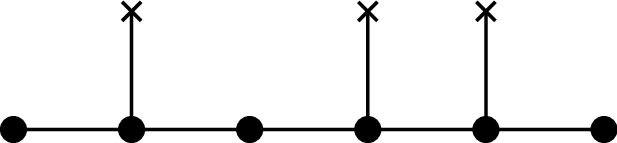}
\caption*{$C_4$}
\end{subfigure}
\begin{subfigure}{0.35\textwidth}
\centering
\includegraphics[scale=0.18]{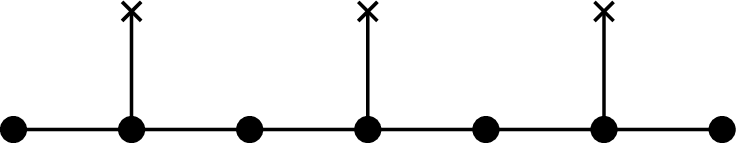}
\caption*{$C_6 \; (\mu\neq 0)$}
\end{subfigure}
\end{figure}

\begin{figure}[H] 
\begin{subfigure}{0.25\textwidth}
\centering
\includegraphics[scale=0.18]{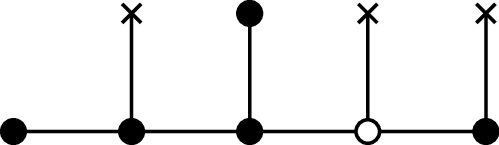}
\caption*{$C_7$}
\end{subfigure}
\begin{subfigure}{0.35\textwidth}
\centering
\includegraphics[scale=0.18]{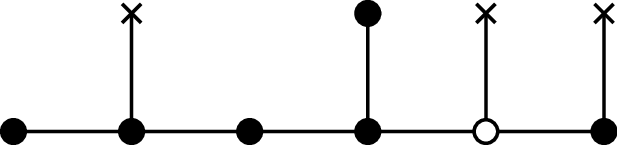}
\caption*{$C_9 \; (\mu\neq 0)$}
\end{subfigure}
\begin{subfigure}{0.35\textwidth}
\centering
\includegraphics[scale=0.18]{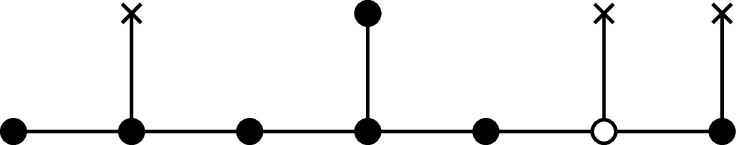}
\caption*{$C_{11} \; (\mu\neq 0)$}
\end{subfigure}

\begin{subfigure}{0.25\textwidth}
\centering
\includegraphics[scale=0.18]{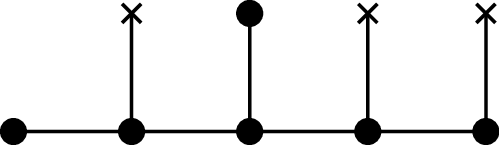}
\caption*{$C_8$}
\end{subfigure}
\begin{subfigure}{0.35\textwidth}
\centering
\includegraphics[scale=0.18]{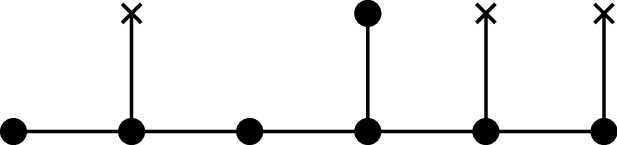}
\caption*{$C_{10} \; (\mu\neq 0)$}
\end{subfigure}
\begin{subfigure}{0.35\textwidth}
\centering
\includegraphics[scale=0.18]{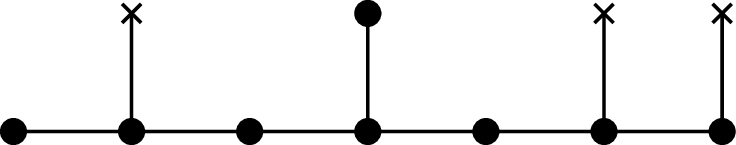}
\caption*{$C_{12} \; (\mu\neq 0)$}
\end{subfigure}
\end{figure}

\begin{figure}[H] 
\begin{subfigure}{0.25\textwidth}
\centering
\includegraphics[scale=0.18]{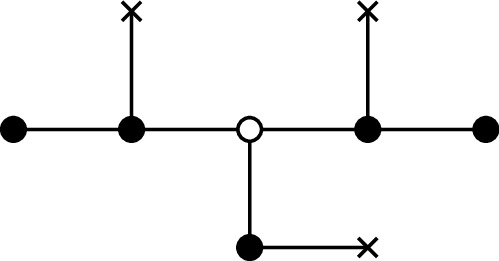}
\caption*{$D_1$}
\end{subfigure}
\begin{subfigure}{0.35\textwidth}
\centering
\includegraphics[scale=0.18]{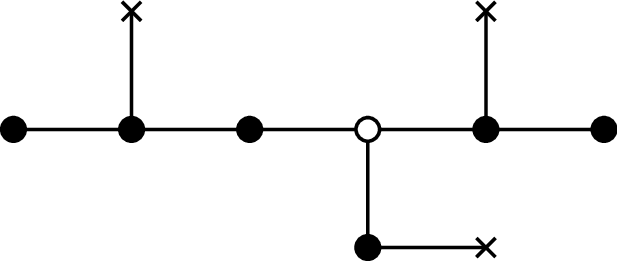}
\caption*{$D_3 \; (\mu\neq 0)$}
\end{subfigure}
\begin{subfigure}{0.35\textwidth}
\centering
\includegraphics[scale=0.18]{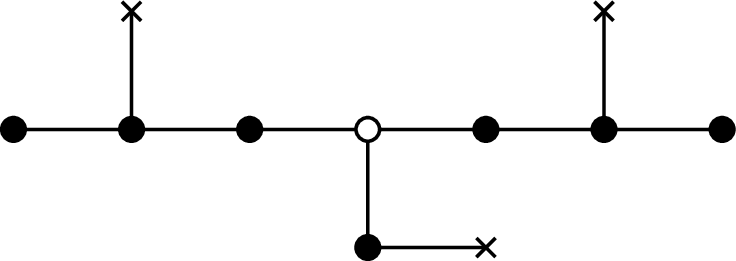}
\caption*{$D_5 \; (\mu\neq 0)$}
\end{subfigure}

\begin{subfigure}{0.25\textwidth}
\centering
\includegraphics[scale=0.18]{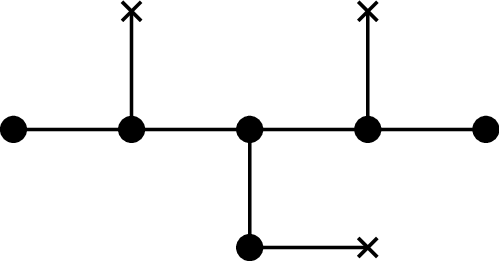}
\caption*{$D_2$}
\end{subfigure}
\begin{subfigure}{0.35\textwidth}
\centering
\includegraphics[scale=0.18]{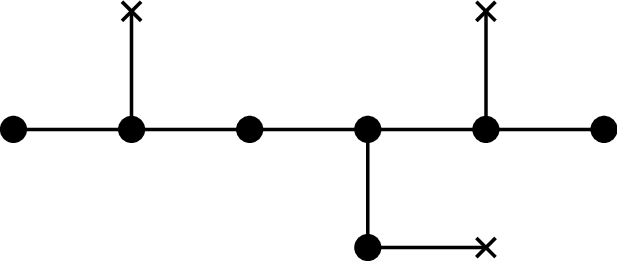}
\caption*{$D_4 \; (\mu\neq 0)$}
\end{subfigure}
\begin{subfigure}{0.35\textwidth}
\centering
\includegraphics[scale=0.18]{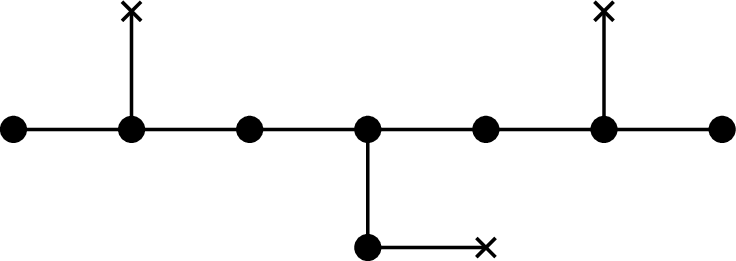}
\caption*{$D_6 \; (\mu\neq 0)$}
\end{subfigure}
\end{figure}

Next, remove all configurations having $\mu \neq 0$, as they are not smoothly contractible.

Finally, observe that some of the remaining configurations are pairwise blow-analytically equivalent (namely, $B_1, B_{11}, C_3, C_7$ and $D_1$ are equivalent to $B_2, B_{12}, C_4, C_8$ and $D_2$ respectively).

Again, we label $\{1, 2, 3\}$ the vertices corresponding to the three non-compact components and consider the values of the invariants $\{\mu(\varGamma_{12}), \mu(\varGamma_{13}), \mu(\varGamma_{23}) \}$ to show that several of the configurations are non-equivalent. In fact, we have the following:
\begin{align*}
B_2, B_{12}, C_4 &: \{1,2,2\} \qquad B_5 :\{0,2,2\} \qquad C_1: \{1,1,2\}\\
 C_8 &: \{0,2,3\} \qquad D_2 : \{1,1,3\}. \qquad \phantom{C_1: \{1,1,2\}}
\end{align*}
To further distinguish between $B_2, B_{12}$ and $C_4$, we look explicitly at the sets $\Delta^*_{ij}$ for $1 \leq i < j \leq 3$:
\begin{align*}
& \Delta^*_{12}(B_2) = \Delta^*_{13}(B_2) = \lbrace \bullet , \bullet , \times \rbrace , \qquad \Delta^*_{23}(B_2) = \lbrace \bullet , \times \rbrace ;\\
& \Delta^*_{12}(B_{12}) = \Delta^*_{13}(B_{12}) = \lbrace \bullet , \bullet , \times \rbrace , \qquad \Delta^*_{23}(B_{12}) = \lbrace \includegraphics[scale=0.14]{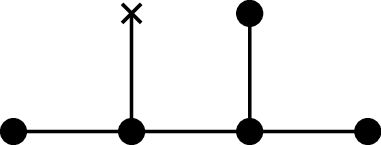} \rbrace; \\
& \Delta^*_{12}(C_4) = \lbrace \bullet , \times \rbrace , \; \Delta^*_{12}(C_4) = \lbrace \bullet , \includegraphics[scale=0.14]{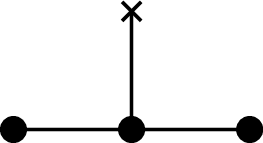} \rbrace , \; \Delta^*_{23}(C_4) = \lbrace \bullet , \bullet , \times \rbrace.
\end{align*}

Since we do not have the bijections implied by Lemma \ref{paths}, we can say that $B_2, B_{12}, C_4$ define different blow-analytic equivalence classes.
\end{proof}

\section{From graphs to germs}

Blow-ups and blow-downs are local transformations, so, in particular, they do not change the order in which the semi-branches intersect the boundary of a small of circle around the origin. We represent this piece of information in a \textit{chord diagram} by drawing vertices on $S^1$ where the semi-branches intersect such boundary, and joining two vertices if they belong to the same local analytic component.

For example, in the tribranched case, there are five possible chord diagrams:
\begin{figure}[H]
\centering
\includegraphics[scale=0.18]{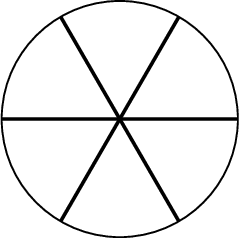}\hspace{5mm}\includegraphics[scale=0.18]{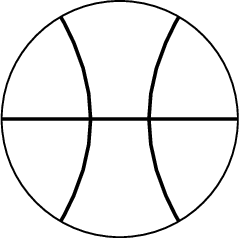}\hspace{5mm}\includegraphics[scale=0.18]{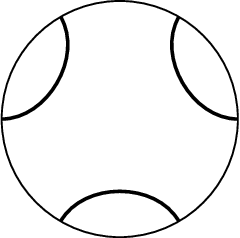}\hspace{5mm}\includegraphics[scale=0.18]{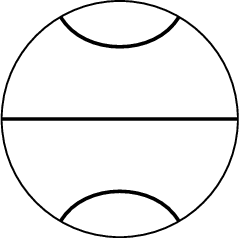}\hspace{5mm}\includegraphics[scale=0.18]{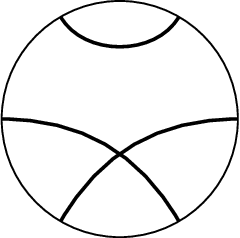}.
\end{figure}

Since chord diagrams are blow-analytic invariants, we can prove that two configurations are non-equivalent by showing that they have different chord diagrams.\\

This invariant does not add new information to the classification of standard forms in Proposition \ref{k1}. In fact, to each configuration corresponds exactly one chord diagram in the following way:
\[
A_2, B_4: \aabccb \qquad B_1, C_2: \aabcbc.
\]

As we consider configurations with a larger value of $\mu'$, however, a new phenomenon appears. For example, the standard form $B_2$ in Proposition \ref{k2}
\begin{figure}[H]
\centering
\includegraphics[scale=0.2]{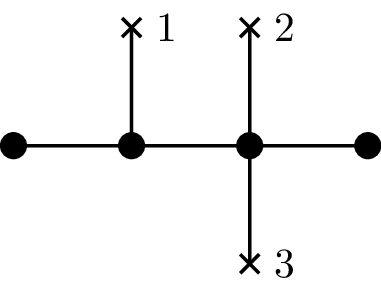}
\end{figure}
\noindent has two possible chords diagrams: \aabbcc and \aabccb. Resolutions corresponding to this dual graph are smooth surfaces diffeomorphic to a chain of four cylinders intersecting transversally, and different choices for the respective positions of the strict transform components give different chord diagrams.

This means that we can have two blow-analytically distinct embedded plane curve germs with the same dual graph, as shown below:

\begin{figure}[H]
\centering
\begin{subfigure}[h]{0.49\textwidth}
\centering
\includegraphics[scale=0.25]{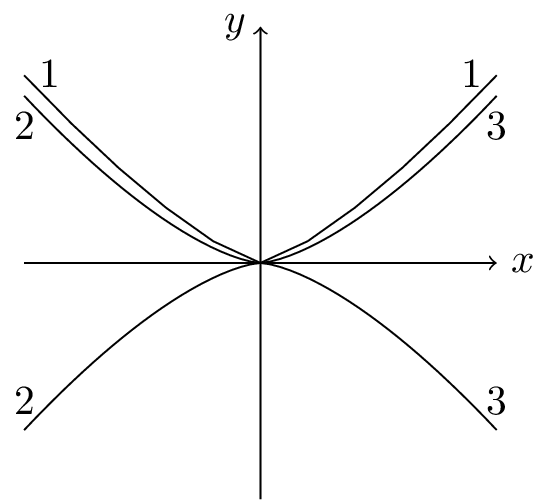}
\caption*{$(\{(y^3-x^4)(y^2+x^3)(y^2-x^3)=0\}, 0)$}
\end{subfigure}
\begin{subfigure}[h]{0.49\textwidth}
\centering
\includegraphics[scale=0.25]{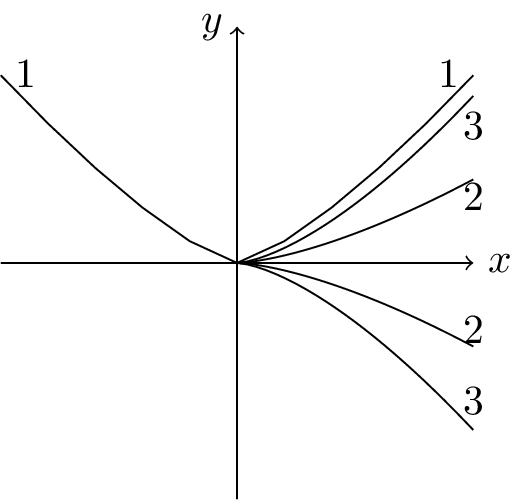}
\caption*{$(\{(y^3-x^4)(2y^2-x^3)(y^2-x^3)=0\}, 0)$}
\end{subfigure}
\end{figure}

The complete list of chord diagrams for the standard forms in Proposition \ref{k2} is as follows:
\begin{align*}
B_2, C_1, C_4, D_2: \aabbcc , \aabccb  \qquad  B_5, C_8: \aabcbc \qquad B_{12}: \aabccb.
\end{align*}

Thus, to each of the configurations $B_2, C_1, C_4, D_2$ correspond pairs of plane curve germs which are blow-analytically non-equivalent.\\

We remark that the induced equivalence of dual graphs is weaker than the blow-analytic equivalence of embedded plane curve germs. This follows from the fact that some topological information is lost in the passage from a resolution to its dual graph, namely, we lose track of the respective position of the strict transform components. One should pay attention to this kind of phenomena when passing from the equivalence of dual graphs $\varGamma$ to the blow-analytic classification of germs.

However, for each standard form, there is only a finite possibility of equivalence classes of germs. In some cases, as above, we can distinguish the classes by using chord diagrams, which are determined solely by the order of the branches near the origin.

\vspace{1cm}

Cristina \textsc{Valle}\\
{Department of Mathematics and Information Sciences,\\ Tokyo Metropolitan University \\
1-1 Minami-Osawa, Hachioji-shi,\\
Tokyo 192-0397, Japan\\
E-mail: \texttt{valle-cristina@ed.tmu.ac.jp}}


\begin{thebibliography}{}
\bibitem{original} Kuo T.-C., \textit{On classification of real singularities}, Invent. Math., \textbf{82} (1985), 257{\---}262.
\bibitem{Kuo} M. Kobayashi, T.-C. Kuo, \textit{On blow-analytic equivalence of embedded curve singularities}, in: Real analytic and algebraic singularities (T. Fukuda et al. eds.), Pitman Research Notes in Math. Series \textbf{381} (1998) 30{\---}37.
\bibitem{Kob} M. Kobayashi, \textit{On Blow-Analytic Equivalence of Branched Curves in $\RR^2$} (preprint).
\bibitem{Koike} S. Koike, A. Parusi\'nski, \textit{Equivalence relations for two variable real analytic function germs}, J. Math. Soc. Japan, \textbf{65} (2013), 237{\---}276.
\bibitem{Fukui} T. Fukui, \textit{Seeking invariants for blow-analytic equivalence}, Compositio Math., \textbf{105} (1997), 95{\---}108.
\bibitem{Nash} J. Nash, \textit{Real algebraic manifolds}, Ann. of Math. \textbf{56}(3) (1952), 405{\---}421.
\bibitem{new} C. Valle, M. Kobayashi, (in preparation).
\end{thebibliography}
\end{document}